\documentclass[12pt,a4paper]{amsart}
\usepackage[latin1]{inputenc}
\usepackage[english]{babel}
\usepackage{amsmath,amssymb,amsthm}

\usepackage{cite}

\theoremstyle{plain}
\newtheorem{theorem}{Theorem}[section]
\newtheorem{lemma}[theorem]{Lemma}

\theoremstyle{definition}
\newtheorem{definition}[theorem]{Definition}

\begin{document}

\title[Markov-Nikolski] {On constrained Markov-Nikolskii type inequality for   
$k-$absolutely monotone polynomials}

\keywords{Markov inequality, Nikolskii inequality, k- monotone polynomial.}

\subjclass[2000]{ 41A17. }

\author[Klurman]{Oleksiy Klurman}

\address{D\'epartment de Math\'ematiques et de Statistique,
Universit\'e de Montr\'eal, CP 6128 succ. Centre-Ville, Montr\'eal QC H3C 3J7, Canada
Canada} \email{\texttt{lklurman@gmail.com}}

\begin{abstract}
We consider a classical problem of estimating norms of higher order derivatives of an algebraic polynomial via the norms of the polynomial itself. The corresponding extremal problem for general polynomials in the uniform norm was solved by V. A. Markov. In $1926,$ Bernstein found the exact constant in the Markov inequality for monotone polynomials. It was shown in \cite{MR2763006} that the order of the constants in constrained Markov-Nikolskii inequality for $k-$ absolutely monotone polynomials is the same as in the classical one in  case $0<p\le q\le\infty.$ In this paper, we find the exact order for all values of $0<p,q\le\infty.$ It turnes out that for the case $q<p$ constrained Markov-Nikolskii inequality is significantly better than the unconstrained one. \end{abstract}

\maketitle

\begin{section}{Introduction}
For $n\ge m\ge 0,$ we denote
$$M_{q,p}(n,m) :=\sup_{P_n\in\mathbb{P}_n}\frac{\|P^{(m)}_n\|_{L_q
[-1,1]}}{\|P_n\|_{L_p [-1,1]}}.$$

In \cite{MR2451401},
 complete information about the orders of $M_{q,p}
(n,m)$ for all values $p,q>0$ is given.
\begin{theorem}\label{glaz1}
  For $0<p,q\le\infty$ we have:
  \begin{equation}\label{eq:ivanov}
M_{q,p} (n,m)\asymp \left\{
\begin{array}{ll}
 n^{2m+2/p-2/q} , & \mbox{\rm if } m>2/q-2/p ,\\
 n^m (\log{n})^{1/q-1/p} , & \mbox{\rm if }  m=2/q-2/p, \\
 n^m, & \mbox{\rm if }  m< 2/q - 2/p .
\end{array}
\right.
\end{equation}

The asymptotic is taken when $m$ is fixed, so the constants may depend on $(m, p, q).$  
\end{theorem}

For each $f\in C[-1,1]$ we denote 
\[\|f\|_{C[-1,1]}:=\|f\|. \]
By $\triangle_n$ we denote the set of all monotone polynomials of
degree $n$ on $[-1,1].$ In $1926,$ S. Bernstein \cite{MR1512353}
pointed out that Markov's inequality for monotone polynomials is not
essentially better than for all polynomials, in the sense, that the
order of $\sup_{P_n\in\triangle_ n}\|P'_n\|/\|P_n\|$ is $n^2$. He
proved his result only for odd $n$. In $2001,$ Qazi \cite{MR1835375}
extended Bernstein's idea to include polynomials of even degree.
Next theorem contains their results:

 \begin{theorem}[Bernstein \cite{MR1512353}, Qazi \cite{MR1835375}] \label{bernquaz}
\[
\sup_{P_n\in\triangle_ n}\frac{\|P'_n\|}{\|P_n\|}=\left\{
\begin{array}{ll}
 \frac{(n+1)^2}{4} , & \mbox{\rm if } n=2k+1 ,\\
 \frac{n(n+2)}{4} , & \mbox{\rm if }  n=2k.
\end{array}
\right.
\]
\end{theorem}

A natural generalization of the concept of monotonicity is $k$-absolute monotonicity.
\begin{definition}
The function $f:[a,b]\to \mathbb{R}$ is {\sl absolutely monotone of
order $k$} if, for all $ x\in [a,b]$,
$$f^{(m)}(x)\ge 0,$$
for all $0\le m\le k$, and denote by $\triangle^{(k)}_n$ the set of
all absolutely monotone polynomials of order $k$ on $[-1,1].$
\end{definition}
For example, absolutely monotone functions of order zero are just
nonnegative functions on $[a,b]$, and
$\triangle_n^{(1)}=\triangle_n\cap\triangle_n^{(0)}$ is the set of
all nonnegative monotone polynomials on $[-1,1].$

A natural modification of $M_{q,p} (n,k)$ for $\triangle^{(k)}_n$ is
$$M_{q,p}^{(k)}(n,m)=\sup_{P_n\in\triangle^{(k)}_n}\frac{\|P^{(m)}_n\|_{L_q
[-1,1]}}{\|P_n\|_{L_p [-1,1]}},$$ for $0\le m\le n,$ $0\le k\le n.$

 In $2009$, A. Kro\'{o} and J. Szabados \cite{MR2564422}  found the exact constants for Markov-Nikolskii inequalities in
$L_1$ and $L_{\infty}$. Note, that  J. Szabados and A. Kro\'{o} referred
to absolutely monotone polynomials of order $k$ as ``$k$-monotone
polynomials.``

The next theorem contains theirs results:

\begin{theorem}[Kro\'{o} and Szabados \cite{MR2564422}, 2009]  For $2\le k\le n$, $m=\left
\lfloor\frac{n-k}{2}\right\rfloor +1$, $\beta
=\frac{1-(-1)^{n-k}}{2}$:

$$M_{\infty,\infty}^{(k)}
(n,1)=\frac{k-1}{1-x_{1,m}^{(k-2,\beta)}},$$
$$M_{1,1}^{(k)} (n,1)=M_{\infty,\infty}^{(k+1)}
(n+1,1),$$
 where $x_{1,m}^{(k-2,\beta)}$ is the largest zero of the Jacobi
 polynomial $J_{m}^{(k-2,\beta)},$  associated with the weight $(1-x)^{k-2}(1+x)^{\beta}.$
\end{theorem}

 T. Erd\'{e}lyi \cite{MR2763006} found the order of  $M_{q,p}^{(k)} (n,m)$ in the case $q\ge
 p$. He was interested in how this order depends on $k$.
 
 \begin{theorem}[Erd\'{e}lyi \cite{MR2763006}, 2009]{\label{erdelyi}}
 For $0\le m\le k/2$, $1\le k\le n$,
$0<p\le
 q\le\infty$, we have \[
 M_{q,p}^{(k)} (n,m) \asymp\left( n^2/k \right)^{m+1/p-1/q}\asymp M_{q,p}(n,m)
 .\]
 First asymptotic in taken when both $n,k\to\infty,$ so the constants depend on $(p,q)$ only. Second asymptotic is taken when $k$ is fixed.
\end{theorem}

It follows from Theorem~\ref{erdelyi} that whenever $q\ge p$ the
order of constants in constrained Markov-Nikolskii inequality
remains the same as in the classical case. In this paper, we find
exact order for all values of $0<p,q\le\infty.$ In particular, the
results imply that the order can be significantly improved when
$q<p.$ Our main result is:

\begin{theorem}{\label{main}}
For $0<p,q\le\infty$ and $p\ne\infty,$ $0< m\le k\le n,$
\[M^{(k)}_{q,p} (n,m)\asymp
\left\{
\begin{array}{ll}
 n^{2m+2/p-2/q} , & \mbox{\rm if } m>1/q-1/p ,\\
 \log^m{n} , & \mbox{\rm if }  m=1/q-1/p, \\
 1 , & \mbox{\rm if }  m< 1/q - 1/p .
\end{array}
\right.
\]
 If $p=\infty,$ $0<q\le \infty$ and $0< m\le k\le n,$ then
\[M^{(k)}_{q,\infty} (n,m) \asymp
\left\{
\begin{array}{ll}
 n^{2m-2/q} , & \mbox{\rm if } m>\frac{1}{q} ,\\
 \log^{m-1}{n} , & \mbox{\rm if }  m=\frac{1}{q}, \\
 1 , & \mbox{\rm if }  m<\frac{1}{q}.
\end{array}
\right.
\]
The asymptotic is taken when $k$ is fixed and so the constants may depend on $(p,q,k).$
\end{theorem}
\end{section}

\begin{section}{Proof of the main result}

{\bf Proof of the upper bound in Theorem \ref{main}.}
We are going to show that for $0<p,q\le\infty,$ $0< m\le k\le n$ and $p\ne\infty,$
\begin{equation}\label{uppermain}M^{(k)}_{q,p} (n,m)\le C(k,p,q)
\left\{
\begin{array}{ll}
 n^{2m+2/p-2/q} , & \mbox{\rm if } m>1/q-1/p ,\\
 \log^m{n} , & \mbox{\rm if }  m=1/q-1/p, \\
 1 , & \mbox{\rm if }  m< 1/q - 1/p .
\end{array}
\right.
\end{equation}

Consider the case $k=1.$ We distinguish between two cases.

{\bf Case 1.} $q\ge 1.$ Clearly, $\frac{1}{q}-\frac{1}{p}< 1.$
Without loss of generality, we can assume that $P_n(-1)=0.$ Note, that for each
$P_n\in\triangle_n^{(1)}$ we have $\|P'_n\|_{L_1[-1,1]}=\|P_n\|.$ By
Nikolskii inequality $$\|P_n'\|_{L_q [-1,1]}\le
C_1(q)n^{2-\frac{2}{q}}\|P'_n\|_{L_1 [-1,1]},$$ and

$$ \|P_n'\|_{L_1[-1,1]} \le \|P_n\|\le C_1(p)n^{\frac{2}{p}}\|P_n\|_{L_p [-1,1]},
$$
so
$$ \|P_n'\|_{L_q[-1,1]} \le C_2(q,p)n^{2-\frac{2}{q}+\frac{2}{p}}\|P_n\|_{L_p
[-1,1]}.
$$

{\bf Case 2.} Let $q<1.$ We first prove, that for all
$P_n\in\triangle_n^1,$ $P_n(-1)=0$ the following inequality holds:
$$\int_{-1}^1 P'^q_n(x)dx\le \frac{1}{q}\int_{-1}^1\frac{P_n^q(x)}{(1-x)^q}dx.$$
Indeed, integration by parts yields
\[S=\int_{-1}^1\frac{P_n^q(x)}{(1-x)^q}dx=\frac{1}{q-1}P_n^q(x)(1-x)^{1-q}|^{1}_{-1}+\frac{q}{1-q}\int_{-1}^1P_n'(x)P_n^{q-1}(x)(1-x)^{1-q}dx.\]
 Since $P_n(-1)=0,$ we have
$$S_1=\frac{1-q}{q}S=\int_{-1}^1P_n'(x)P_n^{q-1}(x)(1-x)^{1-q}dx.$$
We now estimate $S_1+S$ to get the result:
\[S_1+S=\frac{1}{q}S=\int_{-1}^1\left[\frac{P_n^q}{(1-x)^q}+P_n'(x)P_n^{q-1}(x)(1-x)^{1-q}\right]dx\ge \int_{-1}^1\left[P_n'(x) \right]^qdx\]
since \[\frac{P_n^q(x)}{(1-x)^q}+P_n'(x)P_n^{q-1}(x)(1-x)^{1-q}\ge
\left[P_n'(x) \right]^q\] pointwise. Indeed, if \[\frac{P_n^q(x)}{(1-x)^q}\ge
\left[P_n'(x) \right]^q\] the inequality clearly holds. In the other case, if
\[\frac{P_n^q(x)}{(1-x)^q}< \left[P_n'(x) \right]^q,\] then
\[\left[P_n'(x) \right]^{q-1}<(1-x)^{1-q}P_n^{q-1}(x)\] and the second term dominates the right-hand side.

Next we show that it is possible to stay bounded away from the
endpoints of the interval in the sense, that
\[\int_{-1}^1\left[P_n'(x) \right]^qdx\le C_3(q)\int_{-1}^{1-c/n^2}\frac{P_n^q(x)}{(1-x)^q}dx.\]

To prove the last inequality, we estimate
\begin{align*}
\int_{1-c/n^2}^1\frac{P_n^q(x)}{(1-x)^q}dx\le
P^q_n(1)\int_{1-c/n^2}^1(1-x)^{-q}dx
&=\frac{1}{1-q}c^{1-q}n^{2q-2}\|P'_n\|^q_{L_1[-1,1]}\\&\le
c^{1-q}C(q)\|P'_n\|^q_{L_q[-1,1],}
\end{align*}
where the constant $C_1(q)$ comes from the classical Nikoskii
inequality for polynomial $P'_n$ and spaces $L_1[-1,1]$ and
$L_q[-1,1]$ respectively. Taking $c$ to be sufficiently small,
 we can make $c^{1-q}C(q)\le \frac{q}{2}.$ For such defined $c=c(p,q)$ we have 
\begin{equation}{\label{bound1}}
\int_{-1}^1\left[P_n'(x) \right]^qdx\le\frac{2}{q}\int^{1-c/n^2}_{-1}\frac{P_n^q(x)}{(1-x)^q}dx.
\end{equation}

We are ready to prove bounds from above for $k=1.$ For $q\ge p$ the
result follows from the classical Markov-Nikolskii inequality. Let
$\frac{1}{q}=\frac{1}{p}+\frac{1}{r}$ and $r>0.$
 Combining~\eqref{bound1} with Young's inequality we get
\begin{align*}\frac{\|P'_n\|_{L_q [-1,1]}}{\|P_n\|_{L_p [-1,1]}}&\le
\frac{2}{q}\frac{\|P_n(x)(1-x)^{-1}\|_{L_q[-1,1-c/n^2]}}{\|P_n\|_{L_p[-1,1-c/n^2]}}\\&\le\frac{2}{q}\|(1-x)^{-1}\|_{L_r[-1,1-c/n^2].}
\end{align*}
The only thing left is to observe that
\[\|(1-x)^{-1}\|_{L_r[-1,1-c/n^2]}\asymp
\left\{
\begin{array}{ll}
 n^{2+2/p-2/q} , & \mbox{\rm if } 1>1/q-1/p ,\\
 \log{n} , & \mbox{\rm if }  1=1/q-1/p, \\
 1 , & \mbox{\rm if }  1< 1/q - 1/p .
\end{array}
\right.
\]
We prove the upper bound of the theorem for all $k$ by induction. The
base case has been proved above. Let us assume that for each
$P_n\in\triangle_n^{k-1},$ $k\ge 2,$ $1\le m\le k-1$ we have
\[\frac{\|P^{(m)}_n\|_{L_q [-1,1]}}{\|P_n\|_{L_p [-1,1]}}\le C(k-1,q,p)
\left\{
\begin{array}{ll}
 n^{2(m-1)+2/p-2/q} , & \mbox{\rm if } m-1>1/q-1/p ,\\
 \log^{m-1}{n} , & \mbox{\rm if }  m-1=1/q-1/p, \\
 1 , & \mbox{\rm if }  m-1< 1/q - 1/p .
\end{array}
\right.
\]
Take $P_n\in\triangle_n^k.$ If $\frac{1}{q}-\frac{1}{p}=m,$ then
$\frac{1}{q}-\frac{1}{p/p+1}=m-1.$ Using induction hypothesis for
$Q_n=P_n'\in\triangle_n^{k-1},$ we get
\begin{align*}
\frac{\|P^{(m)}_n\|_{L_q [-1,1]}}{\|P_n\|_{L_p [-1,1]}}&\le
C(k-1,q,p/p+1)\log^{m-1} n\frac{\|P'_n\|_{L_{p/p+1}
[-1,1]}}{\|P_n\|_{L_p [-1,1]}}\\&\le C(k,q,p)\log^m n.
\end{align*}
Following the same lines, if $\frac{1}{q}-\frac{1}{p}>m,$ take
$r<\frac{p}{p+1}$ such that $\frac{1}{q}-\frac{1}{r}>m-1$ and use
induction hypothesis to arrive at
\begin{align*}
\frac{\|P^{(m)}_n\|_{L_q [-1,1]}}{\|P_n\|_{L_p [-1,1]}}&\le
C(k-1,q,r)\frac{\|P'_n\|_{L_{r} [-1,1]}}{\|P_n\|_{L_p [-1,1]}}\\&\le
C(k,q,p).
\end{align*}
If $\frac{1}{q}-\frac{1}{p}<m,$ take $r>\frac{p}{p+1}$ such that
$\frac{1}{q}-\frac{1}{r}<m-1$ and use induction hypothesis to get
\begin{align*}
\frac{\|P^{(m)}_n\|_{L_q [-1,1]}}{\|P_n\|_{L_p [-1,1]}}&\le
C(k-1,q,r)n^{2(m-1)+2/p-2/q}\frac{\|P'_n\|_{L_{r}
[-1,1]}}{\|P_n\|_{L_p [-1,1]}}\\&\le C(k,q,p)n^{2m+2/p-2/q}.
\end{align*}
The proof of an upper bound is now complete.

We treat the case $p=\infty$ separately.
\begin{lemma}
\[M^{(k)}_{q,\infty} (n,m)\le C(k,q)
\left\{
\begin{array}{ll}
 n^{2m-2/q} , & \mbox{\rm if } m>\frac{1}{q} ,\\
 \log^{m-1}{n} , & \mbox{\rm if }  m=\frac{1}{q}, \\
 1 , & \mbox{\rm if }  m<\frac{1}{q}.
\end{array}
\right.
\]

\end{lemma}
\begin{proof}
Since $\|P_n\|\ge \|P_n'\|_{L_1[-1,1]}$ the result immediately
follows from ~\eqref{uppermain}.
\end{proof}
To prove the lower bounds we begin with the following two lemmas:
\begin{lemma}{\label{example1}}
Consider
\[Q_n(x)=\sum_{k=1}^n\frac{\alpha(\alpha+1)...(\alpha+k-1)}{k!}x^k\]
for $\alpha=\frac{1}{m}$ and integer $m\ge 1.$ Then
\[\int_{0}^1 Q_n^{\frac{1}{\alpha}}(x)dx\ge C(\alpha)\log n.\]
\end{lemma}
\begin{proof}
For $\alpha=1$ the result immediately follows from direct
integration. For $\alpha\ne 1,$ we first note that
\[\frac{\alpha(\alpha+1)...(\alpha+k-1)}{k!}\sim k^{\alpha-1}.\] Introducing 
\[B_n(x)=1+\sum_{k=1}^nx^kk^{\alpha-1},\]
we are left to show that
\[\int_{0}^1 B_n^m(x)dx\ge C(\alpha)\log n.\]
    Using generalized binomial theorem the coefficient of $x^l,$ $1\le l\le n,$ is equal to 
\[\sum_{l_1+l_2+...+l_m=l}(l_1l_2\cdot...\cdot l_m)^{\alpha-1}.\]
Therefore, since the number of ways to represent $l$ as a sum of $m$ positive integers is equal to ${m+l-1\choose m-1}\sim l^{m-1}$ and each term 
\[(l_1l_2\cdot...\cdot l_m)^{\alpha-1}\ge \left(\frac{l}{m}\right)^{(\alpha-1)m}\sim l^{1-m},\]
 
we get 
\[\sum_{l_1+l_2+...+l_m=l}(l_1l_2\cdot...\cdot l_m)^{\alpha-1}\ge C(\alpha).\]

Therefore
\[\int_{0}^1 B_n^m(x)dx\ge C(\alpha)\int_{0}^1 \sum_{k=0}^n x^kdx  \ge C(\alpha)\log n.\]

\end{proof}
\begin{lemma}
Let $n\in \mathbb{N}$ and
\[Q_{n.\alpha}(x)=\sum_{k=1}^n\frac{\alpha(\alpha+1)...(\alpha+k-1)}{k!}x^k,\]
and $\alpha<1.$ Then
\[\left|Q_{n,\alpha}(x)\right|\le C_1(\alpha),\]
for all $x\in [-1,0].$
\end{lemma}
\begin{proof}
Using Abel's summation formula, we arrive at
\begin{align*}
Q_{n.\alpha}(x)&=\sum_{k=1}^{n-1}\left(\frac{\alpha(\alpha+1)...(\alpha+k-1)}{k!}-\frac{\alpha(\alpha+1)...(\alpha+k)}{(k+1)!}\right)\left(\sum_{i=1}^kx^i\right)\\&+\frac{\alpha(\alpha+1)...(\alpha+n-1)}{n!}\sum_{i=1}^nx^i 
\end{align*}

Observe, that each sum of the form $\sum_{i=1}^kx^i=\frac{1-x^{k+1}}{1-x}\le 2$ for $x\in [-1,0]$ and 
\[\frac{\alpha(\alpha+1)...(\alpha+k-1)}{k!}-\frac{\alpha(\alpha+1)...(\alpha+k)}{(k+1)!}\sim k^{\alpha-2}.\]
The only thin left is to observe that the sum 
\[\sum_{k=1}^{\infty}k^{\alpha-2},\]
converges for $\alpha<1.$ 

\end{proof}

{\bf Proof of the lower bound in Theorem~\ref{main}.} We show that for $0<p,q\le\infty,$ $p\ne\infty$ and $0< m\le k\le n$
\begin{equation}\label{lowermain}M^{(k)}_{q,p} (n,m)\ge C(k,p,q)
\left\{
\begin{array}{ll}
 n^{2m+2/p-2/q} , & \mbox{\rm if } m>1/q-1/p ,\\
 \log^m{n} , & \mbox{\rm if }  m=1/q-1/p, \\
 1 , & \mbox{\rm if }  m< 1/q - 1/p .
\end{array}
\right.
\end{equation}

Note , that in the case when our order is $n^{2m+2/p-2/q}$ the
polynomial was constructed by Erd\'{e}lyi, more precisely, the
construction in Theorem $2.1$ in \cite{MR2763006}  is valid for all
$0<p,q\le\infty.$ So of interest is to construct a polynomial
$P_n\in\triangle_n^k$ such that for all $0<m\le k$
\[M^{(k)}_{q,p}(n,m)\asymp \log^m n.\] By continuity, we can assume that $p,q\in\mathbb{Q}.$ 
Take
\[Q_n(x)=\sum_{k=1}^n\frac{\frac{1}{2mq}(\frac{1}{2mq}+1)...(\frac{1}{2mq}+k-1)}{k!}x^k\]
and consider
\[P_{\nu}(y)=P_{2mn+k}(y)=\int_{-1}^yQ_n^{2m}(x)(y-x)^{k-1}dx.\]
Clearly, $\nu=2mn+k$ and $\deg P_{\nu}=2mn+k.$ It is easy to see,
that $P_{\nu}\in\triangle_{\nu}^k$ and
$P_{\nu}^{(k)}(x)=Q_n^{2m}(x).$ Lemma~\ref{example1} implies
$\|P_{\nu}^{(k)}\|_{L_q[-1,1]}\ge C(k,q)\log^{1/q} n.$ Thus, we are
left to prove that $\|P_{\nu}\|_{L_p[-1,1]}\le C(k,p)\log^{1/p} n.$
Since $P_{\nu}\in\triangle_n^k$ Remez inequality (see \cite{MR1367960}) implies that for
sufficiently small $c=c(p)$
\[\|P_{\nu}\|^p_{L_p[-1,1]}\le 2\int_0^1P^p_{\nu}(x)dx\le C(p)\int_0^{1-c/n^2}P^p_{\nu}(x)dx.\]
Now, for $y>0$ using that $|a+b|^p\le C(p)(|a|^p+|b|^p),$ $y-x\le
1-x$ and $0<Q_n(x)\le (1-x)^{-1/2mq}$ for $x\ge 0$ we can estimate
\begin{align*}
\int_0^{1-c/n^2}P^p_{\nu}(x)dx&=\int_0^{1-c/n^2}\left(\int_{-1}^yQ_n^{2m}(x)(y-x)^{k-1}dx\right)^pdy\\&
\le
C(p)\int_0^{1-c/n^2}\left(\int_{0}^yQ_n^{2m}(x)(y-x)^{k-1}dx\right)^pdy\\&+
C(p)\int_0^{1-c/n^2}\left(\int_{-1}^0Q_n^{2m}(x)(y-x)^{k-1}dx\right)^pdy\\&\le
C(p)\int_0^{1-c/n^2}\left(\int_{0}^yQ_n^{2m}(x)(y-x)^{k-1}dx\right)^pdy\\&+2^kC(p)\left(\int_{-1}^0Q_n^{2m}(x)(y-x)^{k-1}dx\right)^p\\&
\le
\int_0^{1-c/n^2}\left(\int_{-1}^y(1-x)^{-1/q}(1-x)^{k-1}dx\right)^pdy+C_1(p,k,m)\\&\le
C_2(p)\log n.
\end{align*}
It is now straightforward to get a sharp result for all intermediate
derivatives of $k-$ absolutely monotone polynomials by using
\eqref{uppermain} and \eqref{lowermain}.

 The result for the case $p=\infty$ immediately follows from the
 construction described above and the fact $\|P_n\|=\|P_n'\|_{L_1[-1,1]}.$
\end{section}
\\

{\bf Acknowelegment}
\\

The author would like to thank Professor K. Kopotun for all his support and encouragement and to anonymous referee for useful suggestions.\bibliographystyle{plain}

\bibliography{mybliography}

\begin{thebibliography}{1}

\bibitem{MR1512353}
Serge Bernstein.
\newblock Sur l'extension du th\'eor\'eme limite du calcul des probabilit\'es
  aux sommes de quantit\'es d\'ependantes.
\newblock {\em Math. Ann.}, 97(1):1--59, 1927.

\bibitem{MR1367960}
Peter Borwein and Tam{\'a}s Erd{\'e}lyi.
\newblock {\em Polynomials and polynomial inequalities}, volume 161 of {\em
  Graduate Texts in Mathematics}.
\newblock Springer-Verlag, New York, 1995.

\bibitem{MR2763006}
Tam{\'a}s Erd{\'e}lyi.
\newblock A {M}arkov {N}ikolskii type inequality for absolutely monotone
  polynomials of order {$K$}.
\newblock {\em J. Anal. Math.}, 112:369--381, 2010.

\bibitem{MR2451401}
P.~Yu. Glazyrina.
\newblock The sharp {M}arkov-{N}ikolskii inequality for algebraic polynomials
  in the spaces {$L_q$} and {$L_0$} on an interval.
\newblock {\em Mat. Zametki}, 84(1):3--22, 2008.

\bibitem{MR2564422}
A.~Kro{\'o} and J.~Szabados.
\newblock On the exact {M}arkov inequality for {$k$}-monotone polynomials in
  uniform and {$L_1$}-norms.
\newblock {\em Acta Math. Hungar.}, 125(1-2):99--112, 2009.

\bibitem{MR1835375}
Mohammed~A. Qazi.
\newblock On polynomials monotonic on the unit interval.
\newblock {\em Analysis (Munich)}, 21(2):129--134, 2001.

\end{thebibliography}

\end{document}